%% file: APW_TO_Schottky_zero.tex
\documentclass{amsart}

\input{praeambel}
\input{ap_makros}

\usepackage{mathtools}
\usepackage{enumitem}
\setlist[enumerate,1]{label=(\roman*),font=\normalfont}
\input{aa_macros}
\usepackage[colorlinks, linkcolor=blue, breaklinks]{hyperref}
\usepackage{graphics}
\usepackage{enumerate}
\usepackage[T1]{fontenc} 
\usepackage{pdflscape}
\allowdisplaybreaks

\begin{document}

\title{Zero is a resonance of every Schottky surface}
\author[A.\@ Adam]{Alexander Adam}
\address{Institut de Math\'ematiques de Jussieu -- Paris Rive Gauche, Sorbonne Universit\'e, Campus Pierre et Marie Curie, 4, place Jussieu, 75252 Paris Cedex 05, France}
\email{alexander.adam@imj-prg.fr}

\author[A.\@ Pohl]{Anke Pohl}
\address{University of Bremen, Department 3 -- Mathematics, Bibliothekstr.\@
5, 28359 Bremen, Germany}
\email{apohl@uni-bremen.de}

\author[A.\@ Wei\ss{}e]{Alexander Wei\ss{}e}
\address{Max Planck Institute for Mathematics, Vivatsgasse 7, 53111 Bonn, Germany}
\email{weisse@mpim-bonn.mpg.de}

\begin{abstract} 
For certain spectral parameters we find explicit eigenfunctions of transfer operators for Schottky surfaces. Comparing the dimension  of the eigenspace for the spectral parameter zero with the multiplicity of topological zeros of the Selberg zeta function, we deduce that zero is a resonance of every Schottky surface.
\end{abstract}

\maketitle

\section{Introduction and statement of results}

The location of the resonances of hyperbolic surfaces is of great interest for applications, see, e.\,g., \cite{BGS, Bourgain_Kontorovich}. Accordingly much effort has been and is still being spend on their investigation. For the elementary hyperbolic surfaces (i.\,e., those with cyclic\footnote{To improve readibility we restrict here to hyperbolic surfaces in the strict sense, i.\,e., to manifolds, even though many of the results are valid for (two-dimensional real hyperbolic good) orbifolds as well.} fundamental group, namely, the hyperbolic plane, the hyperbolic cyclinders, and the parabolic cyclinders) the precise values of all resonances including their multiplicities are known. For arbitrary hyperbolic surfaces, asymptotic counting results on the number of their resonances such as Weyl laws (for surfaces of finite area), Weyl-type bounds, fractal Weyl bounds and refinements could be established, see e.\,g., \cite{Selberg_Goe,Venkov_book2, Mueller_scattering, GZ_upper_bounds, GZ_scattering_asympt, GZ_Wave, Guillope_Lin_Zworski, Borthwick_sharpbounds, JN, Pohl_Soares}. A good reference for this field is \cite{Borthwick_book}.  Results on precise numerical values for resonances of non-elementary hyperbolic surfaces, however, are still rather sparse. For non-elementary Schottky surfaces $X$ (surfaces which we consider in this note) the only explicitly known resonance is the critical exponent $\delta$ or, equivalently, the Hausdorff dimension of the limit set of the fundamental group of $X$. We provide the exact value of an additional resonance.

\begin{thm}\label{thm:zero}
Zero is a resonance of every Schottky surface.
\end{thm}

For elementary Schottky surfaces, i.\,e., the hyperbolic cylinders, the methods we use for Theorem~\ref{thm:zero} allow us to recover the known result on the full resonance set, albeit with a new proof.

\begin{thm}\label{thm:specialzeros}
The resonance set of the hyperbolic cylinder $C_\ell$ with central geodesic of (primitive) length $\ell$ is $-\N_0 + \frac{2\pi i}{\ell}\Z$. Each resonance has multiplicity $2$.
\end{thm}

The proof of Theorem~\ref{thm:zero} as well as our proof of Theorem~\ref{thm:specialzeros} is based on transfer operator techniques in combination with known properties of the Selberg zeta function. This approach allows us to prove even deeper and more general results than stated in Theorems~\ref{thm:zero} and \ref{thm:specialzeros}, as explained in the following.

Let $X$ be a Schottky surface, let $\mc R(X)$ denote the set of all resonances of $X$ (with multiplicities), let $L(X)$ denote the 
primitive geodesic length spectrum of $X$ (also with multiplicities), and let $\delta$ denote the Hausdorff dimension of the limit set of the fundamental group of $X$. We refer to Section~\ref{sec:prelims} below for more details on the objects used in this overview. 

As is well-known, the Selberg zeta function 
\[
 Z_X(s) \sceq \prod_{\ell\in L(X)}\prod_{k=0}^\infty \left(1 - e^{-(s+k)\ell}\right)
\]
(more precisely, the infinite product used in this definition) converges for $\Rea s >\delta$, and extends to an entire function, which we also denote by $Z_X$. The set of resonances $\mc R(X)$ is contained in the zeros of (the holomorphic continuation of) $Z_X$, respecting multiplicities. Thus, determining zeros of $Z_X$ may serve as a first step towards finding resonances of $X$.

Further, for (any choice of Schottky data of) $X$ there exists a transfer operator family $(\TO_{X,s})_{s\in\C}$ which acts as trace class operators on a certain Hilbert Bergman space and whose Fredholm determinant represents the Selberg zeta function $Z_X$, thus
\[
 Z_X(s) = \det\left(1-\TO_{X,s}\right).
\]
Finding zeros of $Z_X$ is therefore equivalent to determining parameters $s$ for which the transfer operator $\TO_{X,s}$ has an eigenfunction with eigenvalue $1$.

For the case that $X$ is elementary, thus, a hyperbolic cylinder, we can identify not only all such parameters $s$ but also the full eigenspaces of $\TO_{X,s}$.  

\begin{thm}\label{thm:ef_cyl}
Suppose that $X=C_\ell$ is a hyperbolic cylinder with central geodesic of (primitive) length $\ell$. Then $1$ is an eigenvalue of $\TO_{C_\ell,s}$ if and only if $s\in -\N_0 + \frac{2\pi i}{\ell}\Z$. In this case, the geometric multiplicity of the eigenvalue $1$ is $2$. 

For $s=-n+\frac{2\pi i}{\ell}\Z$, $n\in\N_0$, the eigenspace of $\TO_{C_\ell,s}$ with eigenvalue $1$ is essentially spanned by two copies of $z^n$.
\end{thm}

We refer to Theorem~\ref{thm:cyl_full} below for an exact formula for the eigenfunctions.

For the case that $X$ is non-elementary, certain hyperbolic cylinders $C$ provide covers of $X$ (of infinite degree). Even though eigenfunctions of $\TO_{X,s}$ with eigenvalue $1$ are subject to more symmetries than the eigenfunctions of $\TO_{C,s}$ with eigenvalue $1$, for certain choices of $C$ and certain parameters $s$ we can construct eigenfunctions of $\TO_{X,s}$ from those of $\TO_{C,s}$. 

Throughout let $\chi(X)$ denote the Euler characteristic of $X$. 

\begin{thm}\label{thm:ef_arb}
Let $X$ be a Schottky surface and let $s\in -\N_0$. Then $1$ is an eigenvalue of $\TO_{X,s}$ with geometric multiplicity being at least $1-\chi(X)$.  

For $s=-n$, $n\in\N_0$, we can construct $1-\chi(X)$ linear independent eigenfunctions of $\TO_{X,s}$ with eigenvalue $1$ from the eigenfunctions provided by Theorem~\ref{thm:ef_cyl} for certain hyperbolic cylinders covering $X$.
\end{thm}

Theorem~\ref{thm:arb_full} below provides exact formulas for these eigenfunctions. 

To facilitate distinguishing in the following between results known for elementary Schottky surfaces only and those valid for all Schottky surfaces, we denote the (rather, any) hyperbolic cylinder with central geodesic of (primitive) length $\ell$ by $C_\ell$. Surfaces denoted by $X$ refer to arbitrary (elementary or non-elementary) Schottky surfaces. 

Theorem~\ref{thm:ef_cyl} yields that the set of zeros of the Selberg zeta function $Z_{C_\ell}$ for $X=C_\ell$ is $-\N_0+\frac{2\pi i}{\ell}\Z$, each zero having multiplicity $2$. Theorem~\ref{thm:ef_arb} implies that for an arbitrary Schottky surface $X$, the Selberg zeta function $Z_X$ has zeros at $-\N_0$, and that the multiplicity of the zero $s=-n$, $n\in\N_0$, is at least $1-\chi(X)$. 

For completeness we note that for the case $X=C_\ell$ being a hyperbolic cylinder, Theorem~\ref{thm:ef_arb} estimates the multiplicity of the zeros of $Z_{C_\ell}$ at $-\N_0$ from below by $1$ (note that $\chi(C_\ell)=0$), and hence, for the elementary Schottky surfaces,  Theorem~\ref{thm:ef_cyl} is a much stronger result.

For a generic Schottky surface $X$, the zeros of the Selberg zeta function $Z_X$ are not necessarily given only by the resonances of $X$; they may also be topological zeros. These two types of zeros may overlap, causing a zero of $Z_X$ to be of rather high multiplicity. 

Borthwick--Judge--Perry \cite{BJP} showed that (for any Schottky surface $X$) the Selberg zeta function $Z_X$ admits the factorization
\begin{equation}\label{SZF_factorization}
 Z_X(s) = e^{p(s)} G_\infty(s)^{-\chi(X)} P_X(s),
\end{equation}
where $p$ is a certain polynomial of degree at most $2$ (which is of no concern for our application), the function $G_\infty$ is 
\[
 G_\infty(s) = (2\pi)^{-s} \Gamma(s) G(s)^2
\]
with $G$ being the Barnes G-function, and $P_X$ is the entire function
\[
 P_X(s) \sceq s^{m(0)} \prod_{\substack{r\in \mc R(X)\\ r\not=0}} \left(1-\frac{s}{r}\right) \exp\left( \frac{s}{r} + \frac{s^2}{2r^2}\right),
\]
where $m(0)$ denotes the multiplicity of $s=0$ as a resonance of $X$. 

Since the resonances of $X$ are exactly the zeros of $P_X$, including multiplicities, the topological zeros of $Z_X$ are determined (again including multiplicities) by the zeros of $G_\infty(\cdot)^{-\chi(X)}$. The zeros of $G_\infty$ are $s=-n$, $n\in\N_0$, with multiplicity $2n+1$.

Theorem~\ref{thm:zero} then follows by observing that for $s=0$, the geometric multiplicity of the eigenvalue $1$ of $\TO_{X,s}$ exceeds the multiplicity of $s=0$ as topological zero of $Z_X$ by $1$.  

If $X$ non-elementary and hence $\chi(X)\not=0$, and $s < 0$ then such a comparison does not yet provide any insight. However, as shown in \cite{Bunke_Olbrich, Olbrich}, for \emph{non-elementary} Schottky surfaces all zeros at $s\in -\N$ are topological.

Theorem~\ref{thm:specialzeros} follows from combining Theorem~\ref{thm:ef_cyl} with \eqref{SZF_factorization} and the fact that the Euler characteristic of hyperbolic cylinders vanishes, and hence each zero of the Selberg zeta function is indeed a resonance of the same multiplicity.

As mentioned above, the resonance set for hyperbolic cylinders is known for long. The strength of the proof method presented here is not to provide an alternative (rather involved) proof of Theorem~\ref{thm:specialzeros}. Its strength is in establishing Theorem~\ref{thm:ef_cyl} which, in view of the power of transfer operators for dynamical approaches to Laplace eigenfunctions, is an important step towards dynamical and transfer-operator-based characterizations of resonant states as well as period functions for them \cite{Mayer_thermoPSL, Mayer_thermo, Lewis_Zagier,Chang_Mayer_eigen, Chang_Mayer_extension,BLZm, Bruggeman_Muehlenbruch,Pohl_mcf_Gamma0p,Pohl_mcf_general,Moeller_Pohl}.

The structure of this note is as follows. In Section~\ref{sec:prelims} below we provide the necessary background information on hyperbolic surfaces, resonances, Schottky surfaces and transfer operators. The precise formulas for the eigenfunctions of the transfer operators as well as the proofs of Theorems~\ref{thm:zero}-\ref{thm:ef_arb} are the content of Section~\ref{sec:proof_resonance} below.

\textit{Acknowledgement.} AP wishes to thank the Max Planck Institute for Mathematics in Bonn for hospitality and excellent working conditions during the preparation of this manuscript. Further, she acknowledges support by the DFG grant PO 1483/2-1.

\section{Preliminaries}\label{sec:prelims}

In this section we provide the necessary background information on the entities used for the proofs of Theorems~\ref{thm:zero}-\ref{thm:ef_arb}, with the exception of the Selberg zeta function which is already surveyed in the Introduction.

\subsection{Hyperbolic surfaces and resonances}\label{sec:hyperbolic}

We use the upper half plane model 
\[
 \h\sceq \{ z\in\C : \Ima z>0\},\quad ds_z^2 = \frac{dz\,d\overline{z}}{(\Ima z)^2}
\]
for the hyperbolic plane. In this model, the geodesic boundary $\partial\h$ of $\h$ is identified with $\overline\R \sceq \R\cup\{\infty\}$. 

We take advantage of the standard embedding of the hyperbolic plane $\h$ into the Riemann sphere $\wh\C = \C \cup\{\infty\} \cong P^1_\C$. In this embedding, $\partial\h$ is indeed the topological boundary of $\h$.

We identify the group of orientation-preserving Riemannian isometries of $\h$ with $\PSL_2(\R) = \SL_2(\R)/\{\pm \id\}$. If an element $g\in\PSL_2(\R)$ is represented by the matrix
\[
 \mat{a}{b}{c}{c} \in \SL_2(\R)
\]
then we may denote $g$ by 
\[
 g = \bmat{a}{b}{c}{d}.
\]
The action of $\PSL_2(\R)$ extends continuously (even smoothly) to all of $\wh\C$. With all identifications in place, for $g=\textbmat{a}{b}{c}{d}\in\PSL_2(\R)$ and $z\in\wh\C$ we have
\[
 g.z = 
 \begin{cases}
 \infty & \text{if $z=\infty$, $c=0$, or $z\not=\infty$, $cz+d=0$,}
 \\
 \frac{a}{c} & \text{if $z=\infty$, $c\not=0$,}
 \\
 \frac{az+b}{cz+d} & \text{otherwise.}
 \end{cases}
\]

An element $g\in\PSL_2(\R)$ is called \emph{hyperbolic} if it fixes exactly two points in $\partial\h$. The prototypical hyperbolic elements are 
\[
 a_L \sceq \mat{e^{\frac{L}2}}{0}{0}{e^{-\frac{L}{2}}}, \qquad (L>0),
\]
which act on $\h$ as multiplication by $e^L$. Any hyperbolic element in $\PSL_2(\R)$ is conjugate to some (unique) $a_L$ with $L>0$. For any hyperbolic element $h\in\PSL_2(\R)$ we let $x_+(h)$ and $x_-(h)$ denote the attracting and repelling fixed point, respectively, under its iterated action. Thus, for any $z\in\h$ we have
\[
 x_+(h) = \lim_{n\to\infty} h^n.z\quad\text{and}\quad x_-(h) = \lim_{n\to\infty} h^{-n}.z.
\]

Given a Fuchsian group $\Gamma$, that is, a discrete subgroup $\Gamma$ of $\PSL_2(\R)$, we denote the space of its orbits on $\h$ by $\Gamma\backslash\h$. This quotient space inherits from $\h$ the structure of a Riemannian orbifold. If $\Gamma$ is torsion-free then $\Gamma\backslash\h$ is even a Riemannian manifold.

As is well-known, hyperbolic surfaces are precisely those Riemannian manifolds that are isometric to a quotient manifold of the form $\Gamma\backslash\h$ for some torsion-free Fuchsian group $\Gamma$.

The Laplace--Beltrami operator $\Delta_\h$ on $\h$ is given by 
\[
 \Delta_\h\sceq -y^2\left(\partial_x^2 + \partial_y^2\right) \qquad (z=x+iy;\ x,y\in\R, y>0).
\]
For any hyperbolic surface $X$, the Laplace operator $\Delta_\h$ induces the Laplace operator 
\[
 \Delta_X \colon L^2(X) \to L^2(X)
\]
on $X$. Its resolvent
\[
 R_X(s) \sceq \left(\Delta_X - s(1-s)\right)^{-1}\colon L^2(X)\to L^2(X)
\]
is defined for $s\in\C$ with $\Rea s > 1/2$ and $s(1-s)$ not being in the $L^2$-spectrum of $\Delta_X$. It extends to a meromorphic family
\[
 R_X(s)\colon L^2_{\mathrm{comp}}(X) \to H^2_{\mathrm{loc}}(X)
\]
on $\C$, see \cite{Mazzeo_Melrose,GZ_upper_bounds}. The poles of $R_X$ are the \emph{resonances} of $X$. We let 
\[
 \mc R(X)
\]
denote the set of resonances, repeated according to multiplicities.

\subsection{Schottky surfaces and Schottky data}

Schottky surfaces are precisely those hyperbolic surfaces that are convex cocompact and  of infinite area (and without orbifold points). The fundamental group of a Schottky surface is (isomorphic to) a (Fuchsian) Schottky group, that is, a geometrically finite, non-cofinite, torsion-free Fuchsian group without parabolic elements. 

As shown by Button \cite{Button}, every (Fuchsian) Schottky group arises from a certain geometric construction that we recall in what follows, and all Fuchsian groups given by this construction are indeed Schottky. 

To simplify a uniform presentation of the geometric construction we recall that we identify the Riemann sphere $\wh\C$ with $\C\cup\{\infty\}$. We say that a subset $\mc D\subseteq \wh\C$ is a \emph{Euclidean disk} if there exists an element $g\in\PSL_2(\R)$ such that $g.\mc D$ is contained in $\C$  and is a Euclidean disk in the usual sense. One easily checks that the notion of Euclidean disk does not depend on the choice of $g$.

To construct a Schottky group, choose $r\in\N$ and fix $2r$ open Euclidean disks in $\wh\C$ that are centered on $\partial\h$ and have mutually disjoint closures, say
\begin{equation}\label{disks_numbers}
 \mc D_1,\mc D_{-1},\ldots, \mc D_r, \mc D_{-r}.
\end{equation}
For convenience we here deviate from the successive enumeration as used in the standard presentation of the construction. For $j\in\{1,\ldots, m\}$ let $S_j$ be an element of $\PSL_2(\R)$ that maps the exterior of $\mc D_j$ to the interior of $\mc D_{-j}$. Then the subgroup $\Gamma$ of $\PSL_2(\R)$ generated by $S_1,\ldots, S_r$ is Schottky. All the generating element are hyperbolic.

We stress that the numbering of the disks in \eqref{disks_numbers} does not reflect in any way the positions of the disks to each other. In fact, the construction and the arising Schottky group $\Gamma$ depend on the choice of this numbering. Moreover, the constructed Schottky group $\Gamma$ depends on the choice of the elements $\gamma_j$, $j=1,\ldots, r$. For future reference we set
\[
 S_{-j}\sceq S_j^{-1} \qquad (j=1,\ldots, r),
\]
and we refer to the tuple
\[
 \left(r, \big\{\mc D_{\pm j}\big\}_{j=1}^r, \{ S_{\pm j}\}_{j=1}^r\right)
\]
as a \emph{Schottky data} for $\Gamma$. 

To facilitate notation, we use here and throughout
\[
 \big\{ E_{\pm j} \big\}_{j=1}^m \sceq \big\{ E_j, E_{-j} : j=1,\ldots, m\} 
\]
for $E\in\{\mc D, S\}$.

One easily sees that for any fixed Schottky group $\Gamma$, there exist infinitely many choices of Schottky data for $\Gamma$. The number $r$ of generators is an invariant in this construction.

Starting with a Schottky surface $X$, an additional choice in this construction is possible, namely the one of the precise Schottky group $\Gamma$ such that $X$ is isometric to $\Gamma\backslash\h$. We call a tuple 
\[
 \left( \Gamma, r, \big\{\mc D_{\pm j}\big\}_{j=1}^r, \{S_{\pm j}\}_{j=1}^r\right)
\]
a \emph{Schottky data} for $X$ if $\Gamma$ is isomorphic to the fundamental group of $X$, and 
\[
 \left(r, \big\{\mc D_{\pm j}\big\}_{j=1}^r, \{S_{\pm j}\}_{j=1}^r\right)
\]
is a Schottky data for $\Gamma$. As above, $X$ admits infinitely many choices of Schottky data, and the number $r$ of generators of $\Gamma$ is an invariant.

For any Schottky surface $X$ we can find Schottky data such that none of the Euclidean disks $\mc D_{\pm j}$, $j=1,\ldots, r$, contains $\infty$. However, for the discussions in the following sections it is rather convenient to refrain from a restriction to such special Schottky data.

Schottky surfaces $X$ for which the number of generators in a (and hence any) Schottky data is $r=1$ are called \emph{hyperbolic cylinder}. The generator $S_1$ in any Schottky data is conjugate to $a_L$ for some $L>0$. The value of $L$ is independent of the choice of the Schottky data. It coincides with the (primitive) length of the unique non-oriented geodesic on $X$. By slight abuse of notions, we call $X$ (and any hyperbolic surface isometric to $X$) the hyperbolic cylinder with central geodesic of length $L$ and denote it by $C_L$.

\subsection{Transfer operators}\label{sec:TO}

To any Schottky surface $X$ and any choice of Schottky data for $X$ there is associated a transfer operator family which is crucial for our argumentation. In the following we provide a brief presentation of this transfer operator family, its domain of definition and its properties. We refer to \cite{Guillope_Lin_Zworski, Borthwick_book} for proofs and more details. Before we can state the definition of the transfer operator family we need a few preparations. 

For $s\in\C$, any subset $U\subseteq\overline\R$, any element $g = \textbmat{a}{b}{c}{d}\in\PSL_2(\R)$ and any function $f\colon U\to \C$ we set
\[
  \tau_s(g^{-1})f(x) \sceq  \big(g'(x)\big)^s f(g.x) = |cx+d|^{2s} f\left( \frac{ax+b}{cx+d}\right)\qquad (x\in U)
\]
whenever this is well-defined. 

In the following we make use of holomorphic extensions of this definition to functions defined on certain open subsets $U\subseteq\wh C$ and applied to certain subsets of $\PSL_2(\R)$. We stress that such holomorphic extensions are subject to choices, and that none of the possible holomorphic extensions can be extended to a group action of $\PSL_2(\R)$. However, for all applications arising here there is indeed at least one holomorphic extension. The precise possibilities, in particular the choice of the logarithm, heavily depend on the chosen Schottky data. All results obtained are independent of the choice of the holomorphic extension. Since in all situations arising here the possible choices need to obey only finitely many restrictions which are easily seen to be satisfiable, we refrain throughout from this rather tedious discussion and tacitly make appropriate choices. 

After having fixed such choices we say that a function $f\colon U\to \C$ defined on some neighborhood $U$ of $\infty$ is \emph{holomorphic at $\infty$ (for the parameter $s$)}  if there exists  $g\in\PSL_2(\R)$ such that $g.U\subseteq \C$ and the map
\begin{equation}\label{def_holom}
 \tau_s(g)f\colon g.U\to \C
\end{equation}
is holomorphic at $g.\infty$. We remark that for all $g\in\PSL_2(\R)$, the operator $\tau_s(g)$ is \emph{holomorphy-preserving} as long as it is well-defined. We remark further that this notion of holomorphy depends on the value of $s$.

For any open bounded subset $U\subseteq\C$ we let $H^2(U)$ denote the \emph{Hilbert Bergman space} on $U$, that is the space
\[
 H^2(U) \sceq \left\{ \text{$f\colon U\to \C$ holomorphic} : \int_U \|f\|^2 \dvol < \infty  \right\}
\]
of holomorphic square-integrable functions on $U$, endowed with the inner product
\[
 \langle f,g\rangle \sceq \int_U \langle f(z), g(z)\rangle \dvol(z).
\]
We extend this definition to Euclidean disks of $\wh C$ possibly containing $\infty$ by combining it with the definition \eqref{def_holom} of holomorphy at $\infty$ from above. In all our applications, the operator applied to $H^2(U)$ depends on a parameter $s$ which is then used for the natural definition of holomorphy at $\infty$. This has the effect that the space $H^2(U)$ itself effectively does not depend on $s$, only the choice of the transformation map in a manifold chart.

Let $X$ be a Schottky surface and let 
\[
 \left(\Gamma, r,  \{\mc D_{\pm j}\}_{j=1}^m, \{S_{\pm j} \}_{j=1}^m \right)
\]
be a choice of Schottky data for $X$. 

Let 
\[
 \mc H \sceq \bigoplus_{j=1}^m H^2\big(\mc D_j\big) \oplus \bigoplus_{k=1}^m H^2\big(\mc D_{-k}\big)
\]
denote the direct sum of the Hilbert Bergman spaces on the Euclidean disks from the Schottky data. We identify functions $f\in\mc H$ with the function vectors
\begin{equation}\label{coordinates}
 \bigoplus_{j=1}^m f_j \oplus \bigoplus_{k=1}^m f_{-k}
\end{equation}
with 
\[
 f_\ell \in H^2\big( \mc D_\ell \big)\qquad (\ell\in\{\pm 1,\ldots,\pm r\}).
\]

The transfer operator $\TO_s$ with parameter $s\in\C$ associated to $X$ and the Schottky data is the operator 
\[
 \TO_s\colon \mc H\to \mc H
\]
which, with respect to the presentation \eqref{coordinates}, has the matrix presentation
\[
\TO_s = 
\begin{pmatrix} 
\tau_s\big(S_1\big) & \tau_s\big(S_2\big) & \ldots & \tau_s\big(S_r\big) & 0 & \tau_s\big(S_{-2}\big) & \ldots & \tau_s\big(S_{-r}\big)
\\[2mm]
\tau_s\big(S_1\big) & \tau_s\big(S_2\big) & \ldots & \tau_s\big(S_r\big) & \tau_s\big(S_{-1}\big) & 0 & \ldots & \tau_s\big(S_{-r}\big)
\\[2mm]
\vdots & \vdots & \ldots & \vdots & \vdots &  \vdots & \ddots & \vdots
\\[2mm]
\tau_s\big(S_1\big) & \tau_s\big(S_2\big) & \ldots & \tau_s\big(S_r\big) & \tau_s\big(S_{-1}\big) & \tau_s\big(S_{-2}\big) & \ldots & 0
\\[2mm]
0 & \tau_s\big(S_2\big) & \ldots & \tau_s\big(S_r\big) &  \tau_s\big(S_{-1}\big) & \tau_s\big(S_{-2}\big) & \ldots & \tau_s\big(S_{-r}\big)
\\[2mm]
\tau_s\big(S_1\big) & 0 & \ldots & \tau_s\big(S_r\big) & \tau_s\big(S_{-1}\big) & \tau_s\big(S_{-2}\big) & \ldots & \tau_s\big(S_{-r}\big)
\\[2mm]
\vdots & \vdots & \ddots & \vdots & \vdots & \vdots & \ldots & \vdots
\\[2mm]
\tau_s\big(S_1\big) & \tau_s\big(S_2\big) & \ldots & 0 & \tau_s\big(S_{-1}\big) & \tau_s\big(S_{-2}\big) & \ldots & \tau_s\big(S_{-r}\big)
\end{pmatrix}.
\]
This operator is of trace class. Its Fredholm determinant represents the Selberg zeta function 
\[
 Z_X(s) = \det(1-\TO_s).
\]

\section{Proofs of precisions of Theorems~\ref{thm:zero}-\ref{thm:ef_arb}}

In this section we provide the precise formulas for the eigenfunctions as announced in the Introduction, and we present proofs of Theorems~\ref{thm:zero}-\ref{thm:ef_arb}. For convenience we refer to eigenfunctions and eigenspaces with eigenvalue $1$ by $1$-eigenfunctions and $1$-eigenspaces, respectively. 

Each transfer operator considered in this note is build up from operators of the form $\tau_s(h)$, $h\in\PSL_2(\R)$ hyperbolic. For that reason, we first consider, for any $L>0$, the prototypical hyperbolic element
\[
 a_L = \bmat{e^{L/2}}{0}{0}{e^{-L/2}}
\]
with attracting fixed point $\infty$ and repelling fixed point $0$, and study, in Proposition~\ref{babycase:eigenfunction} below, for the operator $\tau_s(a_L)$ the $1$-eigenspaces of functions holomorphic in $0$. With Lemma~\ref{lem:conjugation} we then clarify how these $1$-eigenfunctions relate to $1$-eigenfunctions with corresponding regularity of $\tau_s(h)$ for an arbitrary hyperbolic element $h\in\PSL_2(\R)$. 

Since the transfer operator of a hyperbolic cylinder essentially is the \emph{direct sum} of two of these prototypical operators, the proof of Theorem~\ref{thm:ef_cyl} follows then immediately from Proposition~\ref{babycase:eigenfunction} and Lemma~\ref{lem:conjugation}, see Theorem~\ref{thm:cyl_full} below.

Then we study for which parameters $s$ we can construct from the $1$-eigenfunctions of the transfer operators of hyperbolic cylinders $1$-eigenfunctions for the transfer operators of arbitrary Schottky surfaces. Since the latter transfer operators in general do not enjoy a direct sum structure, the results obtained are less sharp than for hyperbolic cylinders, see Proposition~\ref{prop:vanishing} and Theorem~\ref{thm:arb_full} below. However, for nonpositive integers we can construct explicit $1$-eigenfunctions.

From the results on the existence and counting of $1$-eigenfunctions, the proofs of Theorems~\ref{thm:zero}-\ref{thm:specialzeros} then follow immediately, see Section~\ref{sec:proof_resonance} below.

\subsection{Eigenfunctions for transfer operators for hyperbolic cylinders}

In the following proposition we consider $\tau_s(a_L)$ to act on a space of functions that are defined on a suitable neighborhood of $0$. The result and the proof of the proposition shows that the exact neighborhood does not need to be specified any further as long as it is chosen in such a way that $\tau_s(a_L)$ is well-defined.

\begin{prop}\label{babycase:eigenfunction}
Let $L>0$. Then $\ACT{a_L}$ has a $1$-eigenfunction that is holomorphic in $0$ if and only if $s\in -\N_0 + \frac{2\pi i}{L}\Z$. For $s\in  -n+\frac{2\pi i}{L}\Z$, $n\in\N_0$, the space of such $1$-eigenfunctions is spanned by $z^n$.
\end{prop}
 
\begin{proof}
Let $s\in\C$ and let $f$ be a $1$-eigenfunction for $\ACT{a_L}$ that is holomorphic in $0$. We deduce properties on $s$ and $f$. Holomorphy of $f$ in $0$ yields a power series expansion 
\[
f(z) = \sum_{n=0}^\infty a_n z^n
\]
for suitable\footnote{One can even show that $a_n\in\R$.} $a_n\in\C$, $n\in\N_0$, that is valid in a certain neighborhood $U$ of $0$. It follows that for all $z\in U$, we have
\[
 \ACT{a_L}f(z) = e^{-Ls} \sum_{n=0}^\infty a_n e^{-Ln}z^n = \sum_{n=0}^\infty a_n e^{-L(s+n)}z^n.
\]
Then $f=\ACT{a_L}f$ and uniqueness of power series expansions imply that for all $n\in\N_0$ either $a_n=0$ or $e^{-L(s+n)}=1$. The latter property is equivalent to 
\[
 s\in -n + \frac{2\pi i}{L}\Z.
\]
Thus, since $f$ is a $1$-eigenfunction of $\ACT{a_L}$ and hence $f\not=0$ it follows that $s\in-n+\frac{2\pi i}L\Z$ for some $n\in\N_0$, and $a_m=0$ for $m\in\N_0, m\not=n$, and  $f(z) = a_nz^n$ with $a_n\not=0$. 

Conversely, $f(z)=z^n$ is obviously a holomorphic $1$-eigenfunction for $\ACT{a_L}$ for all $s\in -n+\frac{2\pi i}{L}\Z$.
\end{proof}

The following lemma allows us to transfer the result of Proposition~\ref{babycase:eigenfunction} to any hyperbolic element. Recall from Section~\ref{sec:hyperbolic} that for any hyperbolic element $h\in\PSL_2(\R)$ its  $x_\pm(h)$ denote the attracting and repelling fixed points are denoted by $x_\pm(h)$.

\begin{lemma}\label{lem:conjugation}
Let $s\in\C$. Suppose that $g,h\in \PSL_2(\R)$ are hyperbolic elements for which we find $p\in \PSL_2(\R)$ such that $g=php^{-1}$. Suppose further that $f_h$ is a $1$-eigenfunction of $\ACT{h}$ that is holomorphic at the repelling fixed point $x_-(h)$ of $h$. Then $f_g\sceq \tau_s(p)f_h$ is a $1$-eigenfunction of $\tau_s(g)$ that is holomorphic at $x_-(g)$.
\end{lemma}

\begin{proof}
The statements on the holomorphy follow directly from Section~\ref{sec:TO}. The eigenfunction property is shown by a straightforward calculation, which we provide for the convenience of the reader. We have
\begin{align*}
\ACT{g}f_g & = \ACT{php^{-1}}\ACT{p}f_h = \ACT{p}\ACT{h}f_h = \ACT{p}f_h = f_g.
\end{align*}
This completes the proof.
\end{proof}

The following lemma is crucial for the consideration of arbitrary Schottky surfaces. However, for hyperbolic cylinder it facilitates the further argumentation. Throughout let 
\begin{equation}\label{def_S}
S\sceq \bmat{0}{1}{-1}{0}.
\end{equation}

\begin{lemma}\label{lem:inverse}
Let $h\in \PSL_2(\R)$ be hyperbolic. Then there exists $k\in \PSL_2(\R)$ such that $khk^{-1} = h^{-1}$.
\end{lemma}

\begin{proof}
Since $h$ is hyperbolic, $h$ is conjugate within $\PSL_2(\R)$ to a diagonal element. Thus, there is $L\in\R$ and $p\in G$ such that $php^{-1}=a_L$. From 
\[
 Sa_LS = a_{-L} = a_L^{-1}
\]
it follows that 
\[
 p^{-1}Sphp^{-1}Sp = p^{-1}Sa_LSp = p^{-1}a_L^{-1}p = h^{-1}.
\]
This completes the proof.
\end{proof}

The combination of Proposition~\ref{babycase:eigenfunction} and Lemma~\ref{lem:conjugation} allows us to fully determine the $1$-eigenspaces of the transfer operators of hyperbolic cyclinders. This provides a proof of Theorem~\ref{thm:ef_cyl}, see Theorem~\ref{thm:cyl_full} below. Lemma~\ref{lem:inverse} or rather its proof simplifies the argumentation in the proof of Theorem~\ref{thm:cyl_full}.

\begin{thm}\label{thm:cyl_full}
Let $C_\ell$ be a hyperbolic cylinder and let 
\[
 \left(\Gamma, 2, \{\mc D_1,\mc D_{-1}\}, \{S_1,S_{-1}\}\right)
\]
be a choice of Schottky data for $C_\ell$. Let $p\in\PSL_2(\R)$ be such that 
\[
 S_1 = p a_\ell p^{-1}.
\]
For $s\in\C$ let $\TO_s$ be the associated transfer operator and let $\mc H$ be the Hilbert space as in Section~\ref{sec:TO}. Then $\TO_s$ has a $1$-eigenfunction in $\mc H$ if and only if $s\in -\N_0+\frac{2\pi i}{\ell}\Z$. For $s=-n+\frac{2\pi i}{\ell}\Z$, $n\in\N_0$, the $1$-eigenspace of $\TO_s$ is spanned by the two functions
\[
 \left(\tau_s(p)z^n, 0 \right) \quad\text{and}\quad \left(0,\tau_s(pS)z^n \right)
\]
with $S$ from \eqref{def_S}.
\end{thm}

\begin{proof}
The transfer operator $\TO_s$ for $s\in\C$ is
\[
 \TO_s = 
 \begin{pmatrix}
  \tau_s(S_1) & 0
  \\
  0 & \tau_s(S_1^{-1})
 \end{pmatrix}.
\]
Thus, the $1$-eigenspace of $\TO_s$ is the direct sum of the $1$-eigenspaces of $\tau_s(S_1)$ and $\tau_s(S_1^{-1})$. One easily checks that the element $p\in\PSL_2(\R)$ with the properties as claimed exist. Application of Lemma~\ref{lem:conjugation} and Proposition~\ref{babycase:eigenfunction} (in this order) completes the proof.
\end{proof}

\subsection{Eigenfunctions for transfer operators of arbitrary Schottky groups}

The following proposition is the key result for constructing explicit $1$-eigenfunctions of the transfer operators for arbitrary Schottky surfaces. 

\begin{prop}\label{prop:vanishing}
Let $s\in\C$ and let $h\in \PSL_2(\R)$ be hyperbolic. Suppose that $f$ is a $1$-eigenfunction of $\ACT{h}$ that is holomorphic at $x_-(h)$.
Choose $k\in \PSL_2(\R)$ such that $khk^{-1}=h^{-1}$ (see Lemma~\ref{lem:inverse}).  Let $U\subseteq \wh\C\smallsetminus\{x_\pm(h)\}$ be an open set. Then $\ACT{h}f$ and $\ACT{h^{-1}k}f$ extend holomorphically to $U$. Further
\[
 \left( \ACT{h}f - (-1)^{\Rea s} \ACT{h^{-1}} \ACT{k}f\right)\vert_U \equiv 0\quad\Leftrightarrow\quad \Ima s = 0.
\]
\end{prop}

\begin{proof}
By Lemma~\ref{lem:conjugation} we may restrict without loss of generality to $h=a_L$ for $L>0$ and $k=S=\textbmat{0}{1}{-1}{0}$. Then $U$ is any (non-empty) open subset of $\C\smallsetminus\{0\}$. From Proposition~\ref{babycase:eigenfunction} it follows that $s\in -n+\frac {2\pi i \ell}L$ for some $n\in\N_0$, $\ell\in\Z$, and the eigenfunction $f$  is (up to scaling) $z^n$. Thus, $f$ extends holomorphically to all of $\C\smallsetminus\{0\}$, and hence $\ACT{S}f$ does so. For every $z\in U$ it follows that
\begin{align*}
\left(\ACT{a_L}f-(-1)^{\Rea s}\ACT{a_{-L}}\tau_s(S)f\right)_{|U}(z)& =z^n-z^{n-\frac {4\pi i \ell}L}
= z^n\left( 1 - z^{-\frac{4\pi i \ell}L}\right).
\end{align*}
This difference vanishes for all $z\in U$ if $\ell=0$, and at most for finitely many $z\in U$ if $\ell\ne 0$.  
\end{proof}

In Theorem~\ref{thm:arb_full} below we take advantage of Proposition~\ref{prop:vanishing} for investigating if and how the $1$-eigenfunctions of the transfer operators for hyperbolic cylinders can be inherited, or rather extended, to $1$-eigenfunctions of the transfer operators of arbitrary Schottky surfaces. Theorem~\ref{thm:arb_full} is the announced more refined version of Theorem~\ref{thm:ef_arb}.

\begin{thm}\label{thm:arb_full}
Let $X$ be a Schottky surface and let 
\[
 \left(\Gamma, r, \big\{\mc D_{\pm j}\big\}_{j=1}^r, \big\{ S_{\pm j}\big\}_{j=1}^r \right)
\]
be a choice of Schottky data for $X$. For $s\in\C$ let $\TO_s$ be the associated transfer operator and $\mc H$ be the associated Hilbert space as defined in Section~\ref{sec:TO}. Let $s\in -\N_0$.
\begin{enumerate}
\item Let  $j_0\in\{1,\ldots, r\}$ and let $f_{j_0}$ be a $1$-eigenfunction of $\ACT{S_{j_0}}$ that is holomorphic at $x_-(S_{j_0})$. Fix $k_{j_0}\in \PSL_2(\R)$ such that 
\[
 k_{j_0}S_{j_0}k_{j_0}^{-1} = S_{j_0}^{-1}.
\]
Then $f = (f_1,\ldots, f_r, f_{-1},\ldots, f_{-r})$ with 
\[
 f_{\pm j} \sceq 0 \quad\text{for $j\in\{1,\ldots, r\}\smallsetminus\{j_0\}$}
\]
and
\[
 f_{-j_0} \sceq - (-1)^s\ACT{k_{j_0}} f_{j_0}
\]
defines a $1$-eigenfunction of $\TO_s$ in $\mc H$. 
\item The geometric multiplicity of the eigenvalue $1$ of $\TO_s$ in $\mc H$ is at least $1-\chi(X)$.
\end{enumerate}
\end{thm}

\begin{proof}
Without loss of generality, we may assume that $j_0=1$. Let 
\[
 \wt f \sceq \TO_s f,\qquad
\wt f = (\wt f_1,\ldots, \wt f_r, \wt f_{-1},\ldots, \wt f_{-r}).
\]
Then 
\begin{align*}
 \wt f_1 & = \ACT{S_1}f,
 \\
 \wt f_{-1} & = - (-1)^s\ACT{S_1^{-1}}\ACT{k_1}f_1, 
 \\
 \wt f_j & = \ACT{S_1}f_1 - (-1)^s\ACT{S_1^{-1}k_1}f_1 \quad\text{for $j\in\{\pm2,\ldots, \pm r\}$.}
\end{align*}
By hypothesis, $\wt f_1 = f_1$. By Lemma~\ref{lem:conjugation}, $\wt f_{-1}=f_{-1}$, and by Proposition~\ref{prop:vanishing}, $\wt f_j=0$ for all $j\in\{\pm2,\ldots, \pm r\}$. Thus, $f$ is a $1$-eigenfunction of $\TO_s$. 

By Proposition~\ref{babycase:eigenfunction}, $f_1$ exists, is unique up to scaling, and extends holomorphically to $\wh\C\smallsetminus\{x_+(S_1)\}$. Thus, $f$ extends to an element of $\mc H$. 

Moreover, by letting $j_0$ run through $\{1,\ldots, r\}$ this construction yields $r=1-\chi(X)$ linear independent $1$-eigenfunctions of $\TO_s$. Hence the geometric multiplicity of $1$ as an eigenvalue of $\TO_s$ in $\mc H$ is at least $1-\chi(X)$.
\end{proof}

\subsection{Proofs of Theorem~\ref{thm:zero}-\ref{thm:specialzeros}}\label{sec:proof_resonance}

As already explained in the Introduction, Theorem~\ref{thm:zero} and \ref{thm:specialzeros} now follow from Theorem~\ref{thm:arb_full} and \ref{thm:cyl_full} in comparison with the knowledge on location and orders of topological zeros as obtained from \eqref{SZF_factorization}.

\bibliography{ap_bib}
\bibliographystyle{amsplain}
\end{document}

%% file: praeambel.tex
\newcommand{\apref}[3]{\hyperref[#2]{#1\ref*{#2}#3}}

\usepackage{enumerate}
\usepackage[latin1]{inputenc}
\usepackage{dsfont}
\usepackage{amssymb,amsthm,amsmath}

\input{xy}
\xyoption{all}

\usepackage{mathrsfs}

\theoremstyle{plain}
\newtheorem{prop}{Proposition}[section]
\newtheorem{lemma}[prop]{Lemma}

\newtheorem{thm}[prop]{Theorem}

\theoremstyle{definition}

\theoremstyle{remark}

%% file: ap_makros.tex
\DeclareMathOperator{\TO}{\mc L}

\DeclareMathOperator{\SL}{SL}
\DeclareMathOperator{\PSL}{PSL}

\DeclareMathOperator{\Ima}{Im}
\DeclareMathOperator{\Rea}{Re}

\newcommand\N{\mathbb{N}}

\newcommand\R{\mathbb{R}}
\newcommand\Z{\mathbb{Z}}
\newcommand\C{\mathbb{C}}

\newcommand{\h}{\mathbb{H}}

\newcommand{\mc}[1]{\mathcal #1}

\newcommand{\wt}{\widetilde}
\newcommand{\wh}{\widehat}

\DeclareMathOperator{\dvol}{dvol}

\DeclareMathOperator{\id}{id}

\newcommand{\sceq}{\mathrel{\mathop:}=}

\newcommand{\mat}[4]{\begin{pmatrix} #1&#2\\#3&#4\end{pmatrix}}
\newcommand{\bmat}[4]{\begin{bmatrix} #1&#2\\#3&#4\end{bmatrix}}

\newcommand{\textbmat}[4]{\left[\begin{smallmatrix} #1&#2 \\ #3&#4
\end{smallmatrix}\right]}


%% file: aa_macros.tex
\usepackage{ifthen}
\usepackage{suffix}
\newcommand\be[4][{}]{\ifthenelse{\equal{#1}{}}{\begin{#2}\label{#3}#4\end{#2}}{\begin{#2}[#1]\label{#3}#4\end{#2}}}
\WithSuffix\newcommand\be*[3][{}]{\ifthenelse{\equal{#1}{}}{\begin{#2}#3\end{#2}}{\begin{#2}[#1]#3\end{#2}}}

\usepackage{xspace}

\theoremstyle{definition}

\newlist{enumroman}{enumerate}{1}
\setlist[enumroman,1]{label=(\roman*)}

\newcommand{\se}[2][]{\ifthenelse{\equal{#1}{}}{Section \ref{#2}}{Section \ref{#2:#1}}}

\newcommand{\rem}[2][]{\ifthenelse{\equal{#1}{}}{Remark (\ref{r:#2})}{\tm{#2} \ref{r:#2:#1}}}
\newcommand{\tm}[2][]{\ifthenelse{\equal{#1}{}}{Theorem (\ref{t:#2})}{\tm{#2} \ref{t:#2:#1}}}
\newcommand{\lem}[2][]{\ifthenelse{\equal{#1}{}}{Lemma \ref{#2}}{\lem{#2} \ref{#2:#1}}}
\newcommand{\co}[2][]{\ifthenelse{\equal{#1}{}}{Corollary (\ref{c:#2})}{\co{#2} \ref{c:#2:#1}}}

\newcommand{\ex}[2][]{\ifthenelse{\equal{#1}{}}{Example \ref{ex:#2}}{\ex{#2} \ref{ex:#2:#1}}}
\newcommand{\de}[2][]{\ifthenelse{\equal{#1}{}}{Definition \ref{d:#2}}{\de{#2} \ref{d:#2:#1}}}
\newcommand{\pp}[2][]{\ifthenelse{\equal{#1}{}}{Proposition \ref{#2}}{\pp{#2} \ref{#2:#1}}}

\newcommand{\ACT}[2][s]{\tau_{#1}\left(#2\right)}

%% file: APW_TO_Schottky_zero.bbl
\providecommand{\bysame}{\leavevmode\hbox to3em{\hrulefill}\thinspace}
\providecommand{\MR}{\relax\ifhmode\unskip\space\fi MR }
\providecommand{\MRhref}[2]{%
  \href{http://www.ams.org/mathscinet-getitem?mr=#1}{#2}
}
\providecommand{\href}[2]{#2}
\begin{thebibliography}{10}

\bibitem{Borthwick_sharpbounds}
D.~Borthwick, \emph{Sharp geometric upper bounds on resonances for surfaces
  with hyperbolic ends}, Anal. PDE \textbf{5} (2012), no.~3, 513--552.

\bibitem{BJP}
D.~{Borthwick}, C.~{Judge}, and P.~{Perry}, \emph{{Selberg's zeta function and
  the spectral geometry of geometrically finite hyperbolic surfaces}},
  {Comment. Math. Helv.} \textbf{80} (2005), no.~3, 483--515.

\bibitem{Borthwick_book}
David {Borthwick}, \emph{{Spectral theory of infinite-area hyperbolic surfaces.
  2nd edition}}, 2nd edition ed., Basel: Birkh\"auser/Springer, 2016.

\bibitem{BGS}
J.~Bourgain, A.~Gamburd, and P.~Sarnak, \emph{Generalization of {Selberg's}
  3/16-theorem and affine sieve}, Acta Math. \textbf{207} (2011), no.~2,
  255--290.

\bibitem{Bourgain_Kontorovich}
J.~Bourgain and A.~Kontorovich, \emph{{On {Z}aremba's conjecture}}, {Ann. Math.
  (2)} \textbf{180} (2014), no.~2, 1--60.

\bibitem{BLZm}
R.~{Bruggeman}, J.~{Lewis}, and D.~{Zagier}, \emph{{Period functions for Maass
  wave forms and cohomology}}, {Mem. Am. Math. Soc.} \textbf{1118} (2015),
  iii--v + 128.

\bibitem{Bruggeman_Muehlenbruch}
R.~Bruggeman and T.~M{\"u}hlenbruch, \emph{Eigenfunctions of transfer operators
  and cohomology}, J. Number Theory \textbf{129} (2009), no.~1, 158--181.

\bibitem{Bunke_Olbrich}
U.~{Bunke} and M.~{Olbrich}, \emph{{Group cohomology and the singularities of
  the Selberg zeta function associated to a Kleinian group}}, {Ann. Math. (2)}
  \textbf{149} (1999), no.~2, 627--689.

\bibitem{Button}
J.~Button, \emph{All {Fuchsian Schottky} groups are classical {Schottky}
  groups}, The Epstein Birthday Schrift, Geometry and Topology Publications,
  Coventry, 1998, pp.~117--125.

\bibitem{Chang_Mayer_eigen}
C.-H. Chang and D.~Mayer, \emph{Eigenfunctions of the transfer operators and
  the period functions for modular groups}, Dynamical, spectral, and arithmetic
  zeta functions ({S}an {A}ntonio, {TX}, 1999), Contemp. Math., vol. 290, Amer.
  Math. Soc., Providence, RI, 2001, pp.~1--40.

\bibitem{Chang_Mayer_extension}
\bysame, \emph{An extension of the thermodynamic formalism approach to
  {S}elberg's zeta function for general modular groups}, Ergodic theory,
  analysis, and efficient simulation of dynamical systems, Springer, Berlin,
  2001, pp.~523--562.

\bibitem{Guillope_Lin_Zworski}
L.~{Guillop\'e}, K.~{Lin}, and M.~{Zworski}, \emph{{The Selberg zeta function
  for convex co-compact Schottky groups}}, {Commun. Math. Phys.} \textbf{245}
  (2004), no.~1, 149--176.

\bibitem{GZ_upper_bounds}
L.~Guillop\'{e} and M.~Zworski, \emph{Upper bounds on the number of resonances
  for non-compact {R}iemann surfaces}, J. Funct. Anal. \textbf{129} (1995),
  364--389.

\bibitem{GZ_scattering_asympt}
\bysame, \emph{Scattering asymptotics for {Riemann} surfaces}, Ann. Math.
  \textbf{145} (1997), 597--660.

\bibitem{GZ_Wave}
\bysame, \emph{The wave trace for {R}iemann surfaces}, Geom. Funct. Anal.
  \textbf{9} (1999), no.~6, 1156--1168.

\bibitem{JN}
D.~Jakobson and F.~Naud, \emph{Resonances and density bounds for convex
  co-compact congruence subgroups of $ \mathrm{SL}_{2}(\mathbb{Z}) $}, Israel
  J. Math. \textbf{213} (2000), no.~1, 443--473.

\bibitem{Lewis_Zagier}
J.~Lewis and D.~Zagier, \emph{Period functions for {M}aass wave forms. {I}},
  Ann. of Math. (2) \textbf{153} (2001), no.~1, 191--258.

\bibitem{Mayer_thermo}
D.~Mayer, \emph{On the thermodynamic formalism for the {G}auss map}, Comm.
  Math. Phys. \textbf{130} (1990), no.~2, 311--333.

\bibitem{Mayer_thermoPSL}
\bysame, \emph{The thermodynamic formalism approach to {S}elberg's zeta
  function for {${\rm PSL}(2,{\bf Z})$}}, Bull. Amer. Math. Soc. (N.S.)
  \textbf{25} (1991), no.~1, 55--60.

\bibitem{Mazzeo_Melrose}
R.~Mazzeo and R.~B. Melrose, \emph{Meromorphic extension of the resolvent on
  complete spaces with asymptotically constant negative curvature}, J. Funct.
  Anal. \textbf{75} (1987), no.~2, 260--301.

\bibitem{Moeller_Pohl}
M.~M{\"o}ller and A.~Pohl, \emph{Period functions for {H}ecke triangle groups,
  and the {S}elberg zeta function as a {F}redholm determinant}, Ergodic Theory
  Dynam. Systems \textbf{33} (2013), no.~1, 247--283.

\bibitem{Mueller_scattering}
W.~{M{\"u}ller}, \emph{{Spectral geometry and scattering theory for certain
  complete surfaces of finite volume}}, {Invent. Math.} \textbf{109} (1992),
  no.~2, 265--305.

\bibitem{Olbrich}
M.~{Olbrich}, \emph{{Cohomology of convex cocompact groups and invariant
  distributions on limit sets}}, G\"ottingen: Univ. G\"ottingen, Mathematische
  Fakult\"at (Habil.-Schr.), 2001.

\bibitem{Pohl_mcf_general}
A.~Pohl, \emph{A dynamical approach to {M}aass cusp forms}, J. Mod. Dyn.
  \textbf{6} (2012), no.~4, 563--596.

\bibitem{Pohl_mcf_Gamma0p}
\bysame, \emph{Period functions for {M}aass cusp forms for ${\Gamma}_0(p)$: A
  transfer operator approach}, Int. Math. Res. Not. \textbf{14} (2013),
  3250--3273.

\bibitem{Pohl_Soares}
A.~Pohl and L.~Soares, \emph{Density of resonances for covers of {S}chottky
  surfaces}, arXiv:1807.00299.

\bibitem{Selberg_Goe}
A.~Selberg, \emph{G{\"o}ttingen lectures}, unpublished.

\bibitem{Venkov_book2}
A.~{Venkov}, \emph{{Spectral theory of automorphic functions and its
  applications}}, Dordrecht etc.: Kluwer Academic Publishers, 1990.

\end{thebibliography}
